\documentclass[12pt]{amsart}
\usepackage{amssymb,amscd}
\usepackage{enumerate}
\usepackage[pdftex]{graphics}
\usepackage{graphicx}
\newtheorem{theorem}{Theorem}[section]

\newtheorem{proposition}[theorem]{Proposition}
\newtheorem{lemma}[theorem]{Lemma}
\theoremstyle{definition}

\theoremstyle{remark}

\numberwithin{equation}{section}

\newcommand{\Z}{\mathbb{Z}}
\newcommand{\N}{\mathbb{N}}
\newcommand{\Q}{\mathbb{Q}}

\newcommand{\fa}{{\mathfrak a}}
\newcommand{\fb}{{\mathfrak b}}
\newcommand{\fd}{{\mathfrak d}}
\newcommand{\fe}{{\mathfrak e}}
\newcommand{\fp}{{\mathfrak p}}
\newcommand{\fq}{{\mathfrak q}}

\newcommand{\cO}{{\mathcal O}}

\newcommand{\cS}{{\mathcal S}}

\newcommand{\tQ}{{\widetilde{Q}}}
\newcommand{\lra}{\longrightarrow}

\newcommand{\too}{\longmapsto}

\newcommand{\Cl}{\operatorname{Cl}}
\newcommand{\disc}{\operatorname{disc}}
\newcommand{\smatr}[4]{(\begin{smallmatrix} 
             #1 & #2 \\ #3 & #4 \end{smallmatrix})}

\begin{document}



\title[Pell surfaces]{Arithmetic of Pell surfaces}

\author{S. Hambleton}
\author{F. Lemmermeyer}

\address{Department of Mathematics, The University of Queensland\\
St. Lucia, Brisbane, Queensland, Australia 4072}
\email{sah@maths.uq.edu.au}

\address{M\"orikeweg 1, 73489 Jagstzell, Germany}
\email{hb3@ix.urz.uni-heidelberg.de}

\date{June 29, 2009}

\begin{abstract}
We define a group stucture on the primitive integer points $(A,B,C)$ of 
the algebraic variety $Q_0(B,C)=A^n$, where $Q_0$ is the principal 
binary quadratic form of fundamental discriminant $\Delta$ and 
$n \geq 2$ is fixed. A surjective homomorphism is given from this
group to the $n$-torsion subgroup of the narrow ideal class group of 
the quadratic number field $\Q(\sqrt{\Delta})$.

\end{abstract}



\maketitle

\section{Introduction}
A classical technique for constructing quadratic number fields
with class number divisible by $n$ is studying integral solutions
of the equation
\begin{equation}\label{Yamamoto} 
    X^2-\Delta Y^2=4Z^n, \quad \text{gcd}(X,Z)=1, \quad \Delta \text{ a fundamental discriminant.} 
\end{equation} 
For each integral point $(X,Y,Z)$ we can form the ideal 
$\fa = \bigl( \frac{X+Y\sqrt{\Delta}}2, Z \bigr)$ in the ring of integers
of $\Q(\sqrt{\Delta}\,)$; the ideal $\fa$ has norm $|Z|$ and 
satisfies $\fa^n = \bigl( \frac{X+Y\sqrt{\Delta}}2 \bigr)$, hence generates 
an ideal class of order dividing $n$.

It seems that P. Joubert \cite{Joub} was the first to observe that
a class of prime order $n$ in the group of binary quadratic forms 
with negative discriminant $\Delta$ implies the solvability of the 
equation (\ref{Yamamoto}); Joubert used techniques from the theory
of complex multiplication for exploiting this observation. 
Nagell \cite{Nagell} later used (\ref{Yamamoto}) for proving the 
existence of infinitely many complex quadratic number fields with 
class number divisible by $n$. By extending Nagell's approach, 
Yamamoto \cite{Yama} was able to prove the existence of infinitely 
many {\em real} quadratic number fields with class number divisible by $n$.

The same approach was further extended by various authors; we
mention in particular Craig \cite{Cr73}.

In this article, we interpret (\ref{Yamamoto}) as an affine surface
and show that a certain subset $\cS_n(\Z)$ of the integral points 
on (\ref{Yamamoto}) can be given a group structure in such a way that 
\begin{enumerate}
\item[a)] the integral points on the hyperplane $Z=1$, which lie
          on the Pell conic $X^2 - \Delta Y^2 = 4$, form a subgroup
          with respect to the classical group structure on Pell 
          conics (see \cite{Lem00,Lem03,LemDes3});
\item[b)] there is a surjective group homomorphism 
          $\cS_n(\Z) \lra \Cl^+(K)[n]$ to the $n$-torsion 
          of the narrow class group of $K = \Q(\sqrt{\Delta}\,)$.
\end{enumerate}
These results explain the success of Yamamoto's approach, and at
the same time raise a few new problems that we do not yet fully 
understand. The rational points on the surface lying on the 
hyperplane $Y = 1$ form (the affine part of) a hyperelliptic 
curve $E: X^2 = 4Z^n + \Delta$; in the case $n = 3$, this is 
an elliptic curve. Although the integral points on $E$ do
not form a group in general, it was observed by Buell 
\cite{Bu76,Bu77} and Soleng \cite{Soleng} that the integral 
points on $E$ (and, more generally, certain rational points 
satisfying some technical conditions) give ideal classes of 
order dividing $3$ in such a way that the map from $E$ to the 
$3$-class group respects the group law on the elliptic curve, 
i.e., that collinear points get mapped to classes whose product 
is trivial. B\"olling \cite{Boel} has extended these results
to the hyperelliptic curves lying on the surface (\ref{Yamamoto}).

\begin{figure}[h]
\begin{center}$
\begin{array}{cc}
\includegraphics[width=5.0cm]{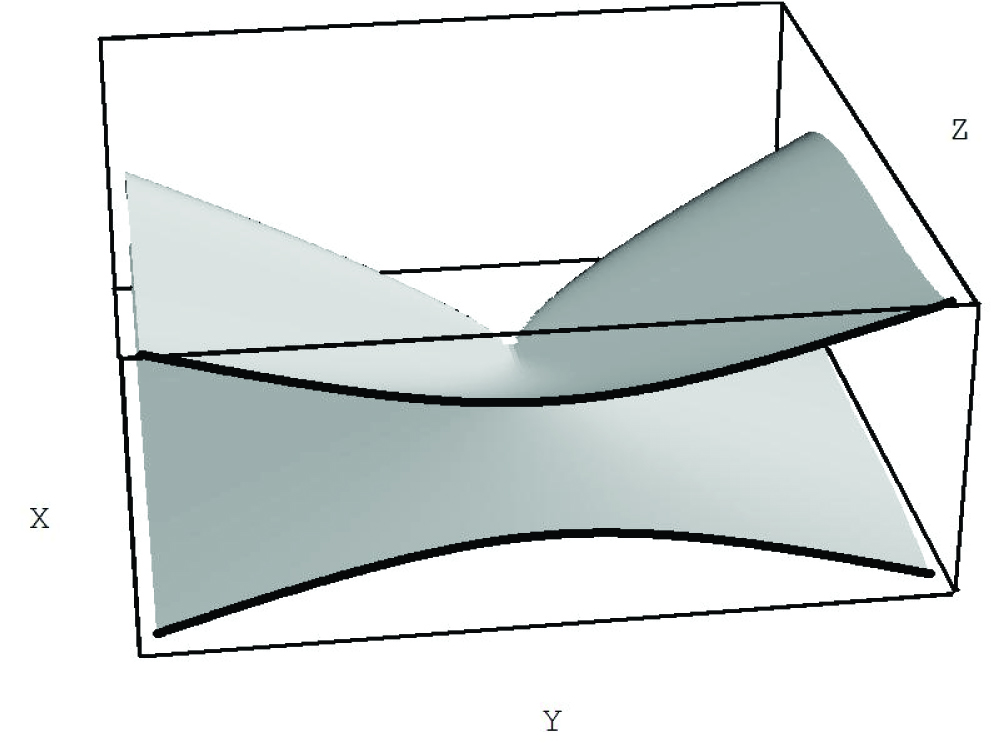}&
\includegraphics[width=5.0cm]{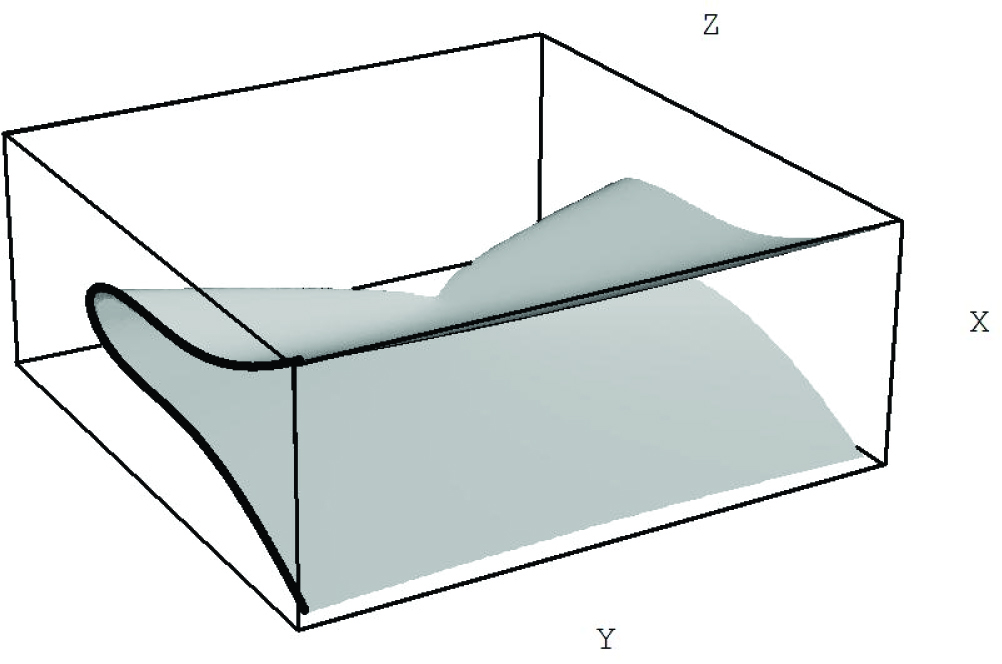}
\end{array}$
\end{center}
\caption{$\cS_3$ for $\Delta >0$ with  cross sections $Z=1$ on the left, a Pell conic, and $Y=1$ on the right, an elliptic curve.}
\label{firstfig}
\end{figure}

\begin{figure}[h]
\begin{center}$
\begin{array}{cc}
\includegraphics[width=5.0cm]{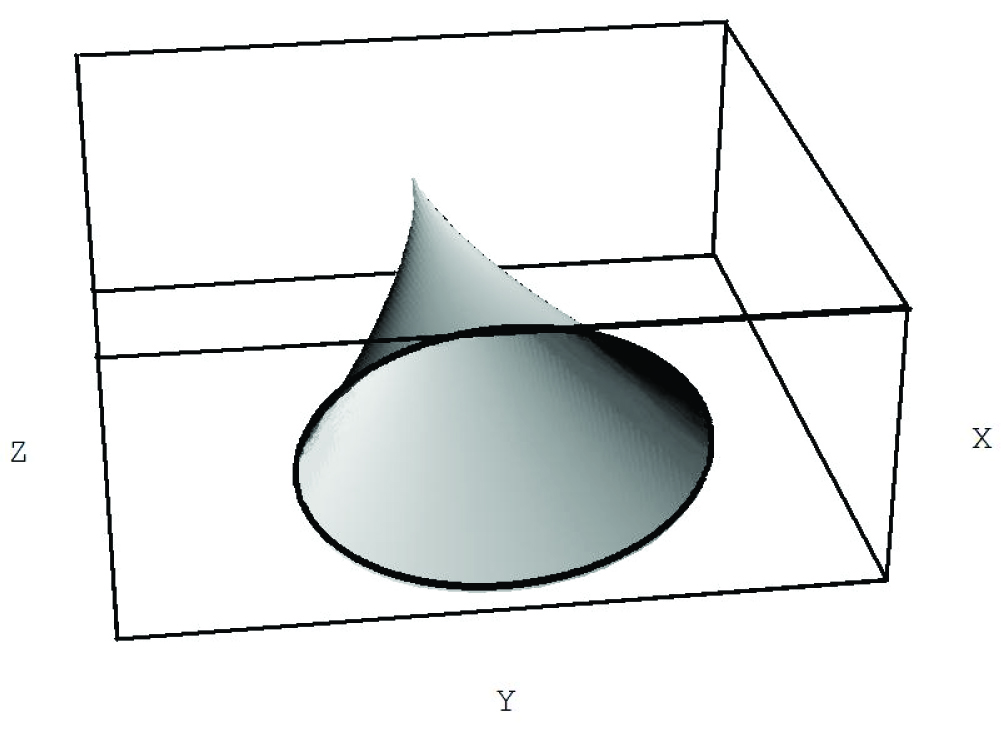}&
\includegraphics[width=5.0cm]{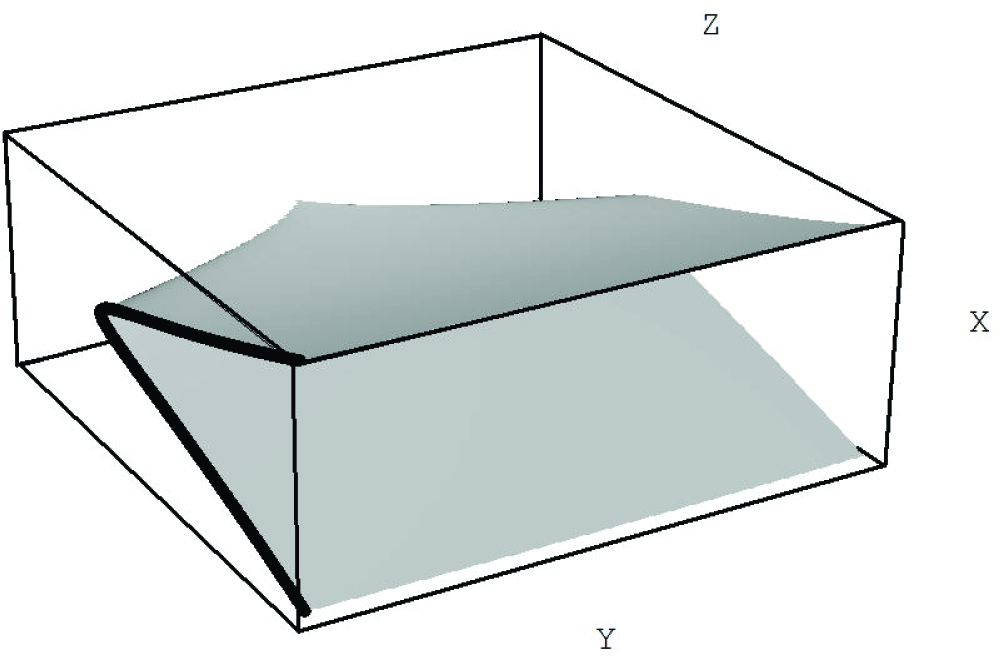}
\end{array}$
\end{center}
\caption{$\cS_3$ for $\Delta <0$ with cross sections $Z=1$ on the left, a Pell conic, and $Y=1$ on the right, an elliptic curve.}
\label{secondfig}
\end{figure}

Although we will see below that the group law\footnote{The group
law on Pell surfaces was discovered by the first author, as was
the homomorphism to the class group.} is best understood by using 
ideals in quadratic number fields, the explicit addition formulas
are tied closely to the composition of binary quadratic forms. 
For this reason, we replace the equation (\ref{Yamamoto}) of the 
surface by $Q_0(X,Y) = Z^n$, where $Q_0$ is the principal form 
with discriminant $\Delta$ defined below. For a brief introduction 
to the composition of binary quadratic forms via Bhargava's cubes 
see Lemmermeyer \cite{LemPep}; more details along more classical 
lines can be found in Flath \cite{Flath}.

\section{Primitive Points on Pell Surfaces}

Let $\Delta$ be a fundamental discriminant (the discriminant of a 
quadratic number field). The principal form with discriminant $\Delta$
is defined by 
$$ Q_0(x,y) = \begin{cases}
         x^2 - my^2       & \text{ if } \Delta = 4m, \\
         x^2 + xy - my^2  & \text{ if } \Delta = 4m+1. \end{cases} $$
By $Q = (a,b,c)$ we denote the binary quadratic form $ax^2 + bxy + cy^2$. 
Such a form $Q$ represents an integer $d$ if $Q(x,y) = d$
for some integers $x, y$; it is said to represent $d$ primitively
if, in addition, $\gcd(x,y) = 1$.

An integral point $(A,B,C)$ on the Pell surface 
\begin{equation}\label{ESn}
   \cS_n: Q_0(B,C) = A^n 
\end{equation} 
is called primitive if $\gcd(B,C) = 1$. The set of primitive 
points on $\cS_n$ will be denoted by $\cS_n(\Z)$.

Now consider Eqn. (\ref{Yamamoto}) and map a point $(A,B,C)$ on the 
Pell surface (\ref{ESn}) to a point $(X,Y,Z)$ on (\ref{Yamamoto}) by 
setting 
$$ (X,Y,Z) = \begin{cases}
              (2B, C, A) & \text{ if } \Delta = 4m, \\
              (2B+C,C,A) & \text{ if } \Delta = 4m+1. 
             \end{cases} $$
This clearly gives a bijection between the integral points on these
surfaces. In addition, Yamamoto's condition $\gcd(X,Z) = 1$ is
easily seen to be equivalent to the primitivity of $(A,B,C)$, that
is, to $\gcd(B,C) = 1$.

\section{The Group Law}

Let $\cO$ denote the ring of integers of the quadratic number field $\Q(\sqrt{\Delta })$. There is a natural map
$\pi_0: \cS_n(\Z) \lra \cO$ defined by 
$\pi_0(A,B,C) = B+C\omega$, where
$$ \omega = \frac{\sigma + \sqrt{\Delta}}2, $$
and $\sigma \in \{0, 1\}$ is defined by $\Delta = 4m+\sigma$.
The elements in the image of $\pi_0$ have the property that 
their norms are $n$-th powers: $N(\pi_0(A,B,C)) = Q_0(B,C) = A^n$.

Consider the set $\cO^*$ of nonzero elements in $\cO$ and
its subset $\N$ of nonzero natural numbers. The set $\cO^*/\N^n$, 
using $\N^n$ to refer to positive integers which are $n$-th powers, 
is a group with respect to multiplication: the neutral element 
is $1\N^n$, the inverse of $\alpha \N^n$ is
$\frac1{\alpha} |N(\alpha)|^n \N^n$ (the element $|N\alpha|/\alpha$
is, up to sign, simply the conjugate $\alpha'$ of $\alpha$, and so
belongs to $\cO^*$). The norm map induces a group homomorphism 
$N: \cO^*/\N^n \lra \Z^*/\Z^{*\,n}$, where $\Z^* = \Z \setminus \{0\}$ 
and $\N$ denote the monoids of nonzero and of positive integers,
repectively.

Observe that if $\alpha, \beta \in \cO^*$ are primitive (this
means that $p \nmid \alpha$ for all primes $p \in \N$) and
$\alpha \N^n \cdot \beta \N^n = \gamma \N^n$, then in general
$\gamma$ cannot be chosen to be primitive. An example is provided
by $\alpha = 3 + \sqrt{3}$ and $\beta = \sqrt{3}$, where
$\gamma = 3 + 3 \sqrt{3}$; here $\gamma \N^n$ is not generated
by a primitive element for any $n \ge 2$.
On the other hand we shall prove below

\begin{proposition}\label{P1}
The cosets of primitive elements in the kernel of the norm map 
$N: \cO^*/\N^n \lra \Z^*/\Z^{*\,n}$ form a subgroup $\Pi_n$ of 
$\cO^*/\N^n$.
\end{proposition}

This fact allows us to prove that there is a bijective map
$\pi: \cS_n(\Z) \lra \Pi_n$ given by $\pi (A,B,C) = (B+C\omega )\N^n$; using this bijection we can 
make $\cS_n(\Z)$ into an abelian group.
The situation is summed up by the following diagram:
$$ \begin{CD}
     \cS_n(\Z) @>{\pi}>{\simeq}> \Pi_n @. @. \\
        @.  @VVV    @.  @. \\ 
    1 @>>> \ker N @>>> \cO^*/\N^n @>{N}>>  \Z^*/\Z^{*\,n}. 
    \end{CD} $$

\begin{theorem}\label{T1}
The map $\pi: \cS_n(\Z) \lra \Pi_n$ is bijective; thus $\cS_n(\Z)$ 
becomes an abelian group by transport of structure.
\end{theorem}

\begin{proof}
Injectivity: assume that there are elements 
$(A,B,C), (A',B',C') \in \cS_n(\Z)$ with $\pi(A,B,C) = \pi(A',B',C')$. 
Then there exist $a, b \in \N$ with 
$(B + C\omega)a^n = (B'+C'\omega)b^n$, and the primitivity of 
$B+C\omega$ and $B'+C'\omega$ implies that $a^n$ and $b^n$ must
be units. Since $a, b \in \N$, this implies $a^n = b^n = 1$.

Surjectivity: assume that $\alpha = B + C\omega$ is primitive with 
$\alpha \N^n \in \Pi_n$. Then $N\alpha = A^n$ for some number 
$A \in \Z^*$ implies $Q_0(B,C) = A^n$, hence 
$(A,B,C) \in \cS_n(\Z)$ with $\pi(A,B,C) = \alpha$.
\end{proof}

Observe that the neutral element of $\cS_n(\Z)$ is the point
$(A,B,C) = (1,1,0)$, and that the inverse of $(A,B,C)$ is given by
$$ -(A,B,C) = \begin{cases}
                 (A,B+\sigma C,-C) & \text{ if } A > 0, \\
                 (A,-B-\sigma C,C) & \text{ if } A < 0. \end{cases} $$
Observe also that the integral points on the Pell conic
$Q_0(T,U) = 1$, which correspond to the points $(1,T,U)$ on the 
Pell surface, inherit their classical group structure since
$(T_1+U_1\omega)(T_2+U_2\omega) = T_3+U_3 \omega$, where
$(T_3,U_3) = (T_1T_2 + mU_1U_2,T_1U_2+T_2U_1)$ if $\Delta = 4m$ and 
$(T_3,U_3) = (T_1T_2 + mU_1U_2,T_1U_2+T_2U_1 + U_1U_2)$ if $\Delta = 4m+1$.
In fact, since the elements $\alpha_j = T_j+U_j\omega$ have norm $1$, 
the element  $\alpha_3 = \alpha_1\alpha_2$ is always primitive.

For proving Prop. \ref{P1} we use the following characterization
of primitive elements:

\begin{lemma}\label{Lprim}
Let $\alpha \in \cO^*$ be a nonzero element of the order $\cO$.
\begin{enumerate}
\item[a)] $\alpha$ is primitive if and only if $(\alpha) + (\alpha') = \fd$ 
      for some ideal $\fd$ dividing the product of all ramified primes. 
\item[b)] If $N \alpha = a^n$ for some $n \ge 2$, then $\alpha$ is primitive 
      if and only if $(\alpha) + (\alpha') = (1)$.
\end{enumerate}
\end{lemma}

\begin{proof}
Assume first that $\alpha$ is primitive, let $\fp$ be an unramified
prime ideal, and set $(\alpha) + (\alpha') = \fd$. If we had 
$\fp \mid \fd$, then $\fp \mid (\alpha)$ and $\fp' \mid (\alpha)$.
Since $\fp$ is unramified, the prime $p$ below $\fp$ either splits (and
then $(p) = \fp\fp'$), or $\fp = (p)$ is inert. In both cases we
deduce that $p \mid (\alpha)$, which contradicts our assumption that 
$\alpha$ be primitive.

Conversely, assume that $\fd$ divides the product of all ramified
primes. If $p \mid \alpha$ for some prime $p \in \N$, then 
$p \mid \alpha'$, hence $p \mid \fd$. This shows that 
$(\alpha) + (\alpha')$ is divisible either by an unramified prime 
ideal or by the square of a ramified prime ideal. This proves the
first statement.

For proving b), assume first that $(\alpha) + (\alpha') = (1)$; 
then $(\alpha)$ is primitive by what we have already proved. 

Finally, if $N\alpha = a^n$ and $\alpha$ is primitive, then 
$\fd$ is a product of ramified prime ideals. But if $\fp || \alpha$ 
for some ramified prime ideal $\fp$ above $p$, then 
$p || \alpha \alpha' = a^n$, and this is impossible for $n \ge 2$.
\end{proof}

\begin{lemma}
Let $\alpha$ be a primitive element. If $\alpha \N^n \in \ker N$, 
then $(\alpha) = \fa^n$ is an $n$-th ideal power. The converse 
holds if $\alpha$ is totally positive.
\end{lemma}

\begin{proof}
The claim is trivial for $n = 1$; assume therefore that $n \ge 2$.

If $\alpha$ is primitive and $N\alpha = a^n$, Lemma \ref{Lprim}
implies that $\alpha$ and $\alpha'$ are coprime. Now 
$(\alpha)(\alpha') = a^n$ implies that $(\alpha) = \fa^n$ is an 
$n$-th ideal power.

Now assume that $(\alpha) = \fa^n$. Then $N\alpha = \pm A^n$ for
some positive integer $A$, and since $\alpha$ is totally positive,
we have $N \alpha > 0$.
\end{proof}

\begin{proof}[Proof of Prop. \ref{P1}]
Assume that $\alpha$ and $\beta$ are primitive elements representing
cosets in the kernel of the norm map. Write $\alpha \beta = \gamma a^n$
with $\gamma \in \cO^*$ and with $a \ge 1$ maximal. We have to show
that $\gamma$ is primitive. 

Assume not; then $p \mid \gamma$ for some rational prime $p$. 
Since $p \nmid \alpha$ and $p \nmid \beta$ (by the primitivity
of these elements), the prime $p$ cannot be inert, and there 
is a prime ideal $\fp$ above $p$ with $\fp \mid (\alpha)$ and 
$\fp' \mid (\beta)$. Since $\alpha$ and $\beta$ are $n$-th ideal
powers, we must have $\fp^{kn} || (\alpha)$ and ${\fp'}^{kn} || (\beta)$,
and this implies $p^{kn} = (\fp\fp')^{kn} || \gamma a^n$. By
the maximality of $a$ we must have $p^{kn} || a^n$, and this 
implies $p \nmid \gamma$.

Thus $\Pi_n$ is closed under multiplication; since the inverse of 
$\alpha \N^n$ is $\pm \alpha' \N^n$ (with the sign chosen in such a way
that $\alpha \cdot (\pm \alpha') > 0$), the set $\Pi_n$ forms a 
subgroup of $\ker N$.
\end{proof}

\medskip\noindent{\bf Remark.} The points with $A=1$ on 
$\cS_n$ form a subgroup of $\cS_n(\Z)$; such points $(1,B,C)$ 
correspond to units $B+C\omega \in \cO$, and the 
group law is induced by the usual multiplication of units.
This shows that the group law on the Pell conic $Q_0(B,C) = 1$
coincides with the standard group law on these curves.

\section{The homomorphism $\cS_n(\Z) \lra \Cl^+(K)[n]$}

We have already remarked that the set $\cS_n(\Z)$ 
was used to extract information on the $n$-torsion of the class 
group $\Cl(K)$ of the quadratic number field $K = \Q(\sqrt{\Delta}\,)$. 
Given a point $(A,B,C) \in \cS_n(\Z)$, we know that $\alpha = B + C\omega$ 
is an $n$-th ideal power: $(\alpha) = \fa^n$. Sending $\alpha$ to the 
narrow ideal class of $\fa$ we get a map $c: \cS_n(\Z) \lra \Cl^+(K)[n]$ from 
$\cS_n(\Z)$ to the group of ideal classes (in the strict sense) in 
$K$ whose order divides $n$:

\begin{proposition}
The map 
$$ c: \cS_n(\Z) \lra \Cl^+(K)[n] $$ 
is a surjective group homomorphism.
\end{proposition}

\begin{proof}
Proving that $c$ is a group homomorphism is easy: let 
$P_j = (A_j,B_j,C_j)$ $\in \cS_n(\Z)$ with $P_1 \oplus P_2 = P_3$,
and put $\alpha_j = B_j + C_j \omega$. Then $(\alpha_j) = \fa_j^n$,
and $\alpha_1 \N^n \cdot \alpha_2 \N^n = \alpha_3 \N^n$
for some $\alpha_3$ that differs from $\alpha_1\alpha_2$
by the $n$-th power of some positive integer $a$. This implies that
$\fa_1^n \fa_2^n = \fa_3^n \cdot a^n$, hence 
$c(P_1)c(P_2) = c(P_1 \oplus P_2)$ as claimed.

For proving that $c$ is onto, consider the narrow ideal class $[\fa] \in \Cl^+(K)[n]$ 
for some ideal $\fa$ coprime to the discriminant. Then 
$\fa^n = (\alpha)$ for some $\alpha = B + C\omega$.
We claim that we can choose $\fa$ in such a way that $\alpha$
is primitive: in fact, let $p$ be a prime dividing $B$ and $C$.
If $p$ is inert, then $\fa = p \fb$, and replacing $\fa$ by $\fb$
does not change the ideal class. If $(p) = \fp\fp'$ is split, then 
we must have $\fp \mid \fa$ and $\fp' \mid \fa$, so again $\fa = p \fb$.
Since $\fa$ is coprime to the discriminant, ramified prime
ideals do not divide $\fa$. Since $(\alpha)$ is principal in the
strict sense, we have $A^n = N\alpha > 0$; writing $\alpha = B + C\omega$
we find $(A,B,C) \in \cS_n(\Z)$ as claimed. 
\end{proof}

Observe that $\cS_n^+(\Z)$, the subset of all $(A,B,C) \in \cS_n(\Z)$
with $A > 0$, forms a subgroup of $\cS_n(\Z)$, and that the proof 
above shows that the natural map $\cS_n^+(\Z) \lra \Cl^+(K)[n]$ is 
surjective.

It is in general difficult to tell whether a point $(A,B,C) \in \cS_n(\Z)$
gives rise to an element of exact order $n$ or not, or more generally,
whether two points generate independent elements. In the following, we
shall briefly recall the criterion used by Yamamoto.

To this end, we introduce a natural homomorphisms between the 
groups $\cS_n(\Z)$:

\begin{proposition}
Assume that $m \mid n$; then there is a group homomorphism
$$ \iota_{m \to n}: \cS_m(\Z) \lra \cS_n(\Z). $$
\end{proposition}

In order to avoid a problematic case, we let $\cS_1(\Z)$ denote
the set of all primitive points $(A,B,C)$ such that $\gcd(A,\Delta) = 1$;
equivalently, $B+C\omega$ is primitive and not divisible by any 
ramified prime ideal.

\begin{proof}
Assume that $(A,B,C) \in \cS_m(\Z)$. With $\alpha = B + C\omega$
we have $(\alpha) = \fa^m$; setting $n = km$, we find 
$(\alpha^k) = \fa^n$, hence $N(\alpha^k) = (A^m)^k = A^n$. 
Observe that $\alpha^k$ is primitive if $\alpha$ is, except
possibly when $m = 1$ and $\alpha$ is divisible by a ramified
prime.

Setting $\alpha^k = B' + C'\omega$, we have $(A,B',C') \in \cS_n(\Z)$.
Since the map $\iota_{m \to n}$ sending $(A,B,C)$ to $(A,B',C')$ is 
compatible with the group structure (in fact: if 
$(B_1 + C_1 \omega) a_1^m \cdot (B_2 + C_2 \omega) a_2^m
  = (B_3 + C_3 \omega) a_3^m$, then raising this equation to the
$k$-th power shows that 
$(B_1' + C_1' \omega) a_1^n \cdot (B_2' + C_2' \omega) a_2^n
  = (B_3' + C_3' \omega) a_3^n$), the claim follows.
\end{proof}

As an example, consider the surface $B^2 + BC + 6C^2 = A^3$;
using the point $(6,1,-1)$ on $\cS_1(\Z)$ we find
$(1-\omega)^3 = -11 + 5\omega$, which gives us the point
$(6,-11,5) \in \cS_3(\Z)$.

It is desirable to have criteria for deciding whether a point
$P \in \cS_n(\Z)$ is actually a newpoint, i.e., does not
come from $\cS_m(\Z)$ for some proper divisor $m$ of $n$.

\begin{proposition}
Assume that $\Delta < -4$. If $P = (A,B,C) \in \cS_n(\Z)$ and 
$n = mp$ for some odd prime $p$, then $P = \iota_{m \to n}(Q)$ 
for some $Q \in \cS_m(\Z)$ implies that $2B + \sigma C$ is a 
$p$-th power modulo $q$ for every prime $q \mid A$.
\end{proposition}

\begin{proof}
Let $\alpha = B + C\omega$ and $(\alpha) = \fa^n$.
If $P =  \iota_{m \to n}(Q)$ for some $Q \in \cS_m(\Z)$,
then $\fa^m = (\beta)$ for $\beta = b+c\omega$ and 
$Q = (A,b,c)$. Thus $\alpha = \pm \beta^p = (\pm \beta)^p$
is a $p$-th power. Let $q$ be a prime dividing $A$; then 
$(q) = \fq\fq'$ splits in $k$, and we have $\beta \in \fq'$ 
and $\beta' \in \fq$.

If $\Delta = 4m$, then $b \equiv c\sqrt{m} \bmod \fq$,
hence $\beta = b+c\sqrt{m} \equiv 2b \bmod \fq$,
$\alpha = B + C\sqrt{m} \equiv 2 B \bmod \fq$, and so
$2B \equiv \alpha = \beta^p \equiv (2b)^p \bmod \fq$.
This implies $2B \equiv (2b)^p \bmod q$ as claimed.

Now assume that $\Delta = 4m+1$. Then $b+c\omega' \in \fq$ shows
that $b+c \equiv c\omega \bmod \fq$ (since $\omega  \omega' = 1$),
hence $2B+C \equiv B + C\omega = (b+c\omega)^p \equiv (2b+c)^p \bmod q$.
\end{proof}

This criterion is not very strong; it does not detect that the
points $(2,1,1)$ or $(3,1,2)$ on $\cS_3: B^2 + BC + 6C^2 = A^3$
are newpoints. On the other hand, $(13,37,6)$ must be a newpoint
since $80 = 2 \cdot 37 + 6$ is not a cube modulo $13$.

\section{Explicit Formulas}\label{SFor}

Let us now make the group law on $\cS_n(\Z)$ explicit by deriving 
addition formulas
\begin{equation}\label{EABC}
  (A_1,B_1,C_1) \oplus (A_2,B_2,C_2) = (A_3,B_3,C_3). 
\end{equation} 
From the definition of the group law it is clear that such addition
formulas must involve computations of greatest common divisors.
The following lemma contains the technical part of the proof:

\begin{lemma}\label{Lgcd}
For points $(A_j,B_j,C_j) \in \cS_n(\Z)$, $j \le 3$, we set
$\alpha_j = B_j + C_j \omega$. Let $\fd = (\alpha_1,\alpha_2')$;
then $\fd = \fe^n$ is an $n$-th ideal power, and with $e = N\fe$,
we have 
\begin{equation}\label{Egcdab}
   \gcd(B_1B_2 + mC_1C_2, B_1C_2 + B_2 C_1 + \sigma C_1C_2) = e^n.
\end{equation}
Conversely, the gcd on the left hand side of (\ref{Egcdab})
is an $n$-th power, and if  (\ref{Egcdab}) holds, then 
$(\alpha_1,\alpha_2') = \fe^n$ for an ideal $\fe$ with norm $e$.
\end{lemma}

\begin{proof}
Since $(\alpha_j) = \fa_j^n$, the ideal $\fd$ must be an $n$-th power.
From $\fe \mid \fa_1$ and $\fe \mid \fa_2'$ we deduce that 
$(e^n) = (\fe\fe')^n \mid (\alpha_1\alpha_2)$, and now
\begin{equation}\label{Eal12}
 \alpha_1 \alpha_2 = (B_1+C_1\omega)(B_2+C_2\omega)
    = B_1B_2 + C_1C_2m + (B_1C_2 + B_2C_1 + \sigma C_1C_2)\omega 
\end{equation}
implies $e^n \mid \gcd(B_1B_2 + C_1C_2m, B_1C_2 + B_2C_1 + \sigma C_1C_2)$.

If, conversely, $p$ is a prime dividing 
$d = \gcd(B_1B_2 + C_1C_2m,B_1C_2 + B_2C_1 + \sigma C_1C_2)$, then the
primitivity of $P_j$ implies that $(p) = \fp\fp'$ must be split in 
$K = \Q(\sqrt{\Delta}\,)$. If, say, $\fp \mid \alpha_1$, then the
primitivity of $\alpha_1$ shows that we must have $\fp' \mid \alpha_2$
and therefore $\fp' \mid \alpha_1'$. Thus if $p^m$ is the exact
power of $p$ dividing $d$, then $\fp^m$ is the exact power of $\fp$
dividing $\alpha_1$, and the fact that $(\alpha_1)$ is an $n$-th ideal 
power shows that $m$ must be a multiple of $n$. This implies that 
\begin{itemize}
\item $d = e^n$ must be an $n$-th power, 
\item $(e) = \fe \fe'$ is the norm of an ideal $\fe$, 
\item and $\fe^n \mid (\alpha_1,\alpha_2')$.
\end{itemize}
This completes the proof.
\end{proof}

Now we can present the explicit formulas for adding points on $\cS_n(\Z)$:

\begin{proposition}\label{Pexpl}
For $(A_1,B_1,C_1), (A_2,B_2,C_2) \in \cS_n(\Z)$ we have
the addition formula (\ref{EABC}), where 
$$ A_3 = \frac{A_1A_2}{e^2},  \quad
   B_3 = \frac{B_1B_2 + mC_1C_2}{e^n}, \quad
   C_3 = \frac{B_1C_2 + B_2C_1 + \sigma C_1C_2}{e^n},$$
with $e$ as in (\ref{Egcdab}).
\end{proposition}

\begin{proof}
The group law is defined via 
$\alpha_1 \N^n  \cdot \alpha_2 \N^n = \alpha_3 \N^n$, where
$\alpha_3$ is required to be primitive. Equation (\ref{Eal12})
and Lemma \ref{Lgcd} show that $\alpha_3 = \alpha_1\alpha_2/e^n$.
Taking norms yields $A_3^n = A_1^nA_2^n/e^{2n}$, and this proves 
the claim.
\end{proof}

\begin{figure}[tb]
\includegraphics[width=5.5cm]{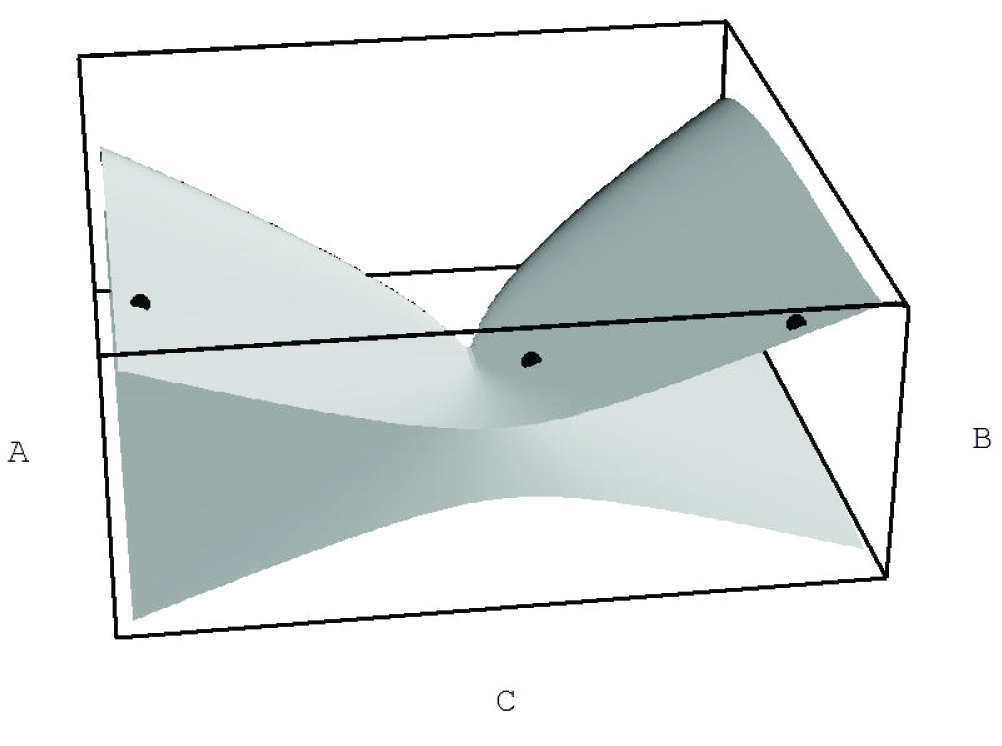}
\caption{From left to right, $(3,92,13)\oplus (3,17,-2) \oplus (9,93,-11) = (1,1,0)$ on $B^2+BC-57C^2=A^3$}
\label{firstfig2}
\end{figure}

\section{From Points to Forms}

Since there is a bijection between ideal classes and equivalence 
classes of binary quadratic forms, we can also describe the group
law in terms of forms. It turns out that the geometric aspects of
the description of the group law on $\cS_n(\Z)$ in terms of
forms adds a lot to our understanding of the arithmetic of Pell
surfaces and Pell conics. For this reason, we will now construct
a map sending primitive points on $\cS_n(\Z)$ to primitive
quadratic forms with discriminant $\Delta$.

Given $(A,B,C) \in \cS_n(\Z)$, consider the form
$$ \tQ_P = (A,2B+\sigma C,A^{n-1}). $$
In order to get positive definite forms if $\Delta < 0$ we 
now agree to replace $\cS_n(\Z)$ by $\cS_n(\Z)^+$ in this case.
It is easily checked that $\disc \tQ_P = \Delta C^2$; moreover,
Dirichlet composition immediately shows that $\tQ_P^n$ is the
principal form (with discriminant $\Delta C^2$). For 
constructing a form with discriminant $\Delta$, we have to 
``underive'' $\tQ_P$. This process replaces a form $(a,b,c)$
with discriminant $\Delta C^2$ by an equivalent form 
$(a',b'C,c'C^2)$, and then maps it to $Q_P = (a',b',c')$, which 
is a primitive form with discriminant $\Delta$. Mapping 
$P \in \cS_n(\Z)$ to the equivalence class of the form $Q_P$ 
turns out to be a homomorphism  $\cS_n(\Z) \lra \Cl^+(\Delta)[n]$. 

Underiving $\tQ_P$ is accomplished by changing the middle 
coefficient modulo $2A$ in such a way that it becomes a multiple 
of $C$. For motivating the following lemma, consider the
equation $2B + 2Ak = 2\beta C$; dividing through by $2$
and reducing mod $A$ yields $\beta C \equiv B \bmod A$, and
this congruence has a unique solution. In this way we find

\begin{lemma}
Given a point $P = (A,B,C) \in \cS_n(\Z)$ with $B^2 - 4AC = \Delta C^2$, 
let $\beta$ be an integer satisfying the congruence 
$\beta \equiv \frac BC \bmod A$; then $\beta^2 \equiv \Delta \bmod A$.
Define a quadratic form $Q_P = (A,2\beta+\sigma,\gamma)$ with
$\gamma = Q_0(\beta,1)/A$. Then $Q_P$ is a primitive form with 
discriminant $\Delta$, and $Q_P$ is positive definite if $\Delta < 0$.
\end{lemma}

\begin{proof}
The claim concerning $\beta$ follows easily from  
$\beta^2 \equiv \frac{B^2}{C^2} \equiv \frac{B^2-4AC}{C^2} = \Delta 
 \bmod A$.

Assume now that $\Delta = 4m$, and set $A = (A,2B,A^{n-1})$. 
From $\beta \equiv \frac BC \bmod A$
we see that there is an integer $k$ with $\beta C = B + Ak$.
Setting $S = \smatr{1}{k}{0}{1}$ we find $Q' = Q|_S = (A,2B',C')$
with $2B' = 2B + 2Ak = 2\beta C$; the integer $C'$ is determined
by $(2B')^2 - 4AC' = \Delta C^2$, which gives
$C' = \frac{\beta^2-m}{A} C^2$. Setting $\gamma = \frac{\beta^2-m}{A}$,
the form $Q_1 = (A,2\beta,\gamma)$ is primitive, has discriminant 
$\Delta$, and the fact that $A > 0$ implies that $Q_1$ is positive
definite if $\Delta < 0$.

The proof in the case $\Delta = 4m+1$ is analogous; here we find
$\gamma = \frac{\beta^2+\beta-m}A$.
\end{proof}

Sending $P \too Q_P$ defines a map $b: \cS_n(\Z) \lra \Cl^+(\Delta)$
between two abelian groups; we already know that the corresponding
map to the ideal class group is a homomorphism, and of course the
same holds for form classes. We will check the details below; now 
let us determine the kernel of $b$. To this end, recall how we 
constructed $b$: to a point $(A,B,C) \in \cS_n(\Z)$
we have attached a quadratic form $\tQ_P = (A,2B+\sigma C,A^{n-1})$
with discriminant $\Delta C^2$; this form $\tQ_P$ is equivalent
to a form $Q_P' = (A,(2\beta+\sigma)C,\gamma C^2)$, and underiving
$Q_P'$ gave us $Q_P = (A,2\beta+\sigma,\gamma)$.
 
Now a point $P = (A,B,C) \in \cS_n(\Z)$ is in the kernel if and only 
if $Q_P \sim Q_0$, which happens if and only if $Q_P$ represents $1$.
Multiplying $Q_P(x,y) = 1$ through by $C^2$ this shows that 
$Q_P'(Cx,y) = C^2$ (conversely, this equation implies $Q_P(x,y) = 1$).
But $Q_P'$ represents $C^2$ properly if and only if the equivalent
form $\tQ_P$ does. Thus we have shown

\begin{proposition}
The kernel of the map $b: \cS_n(\Z) \lra \Cl^+(\Delta)$
consists of all points $(A,B,C) \in \cS_n(\Z)$ with the 
following property: there exist coprime integers $T$, $U$
such that $AT^2 + (2B+\sigma C)TU + A^{n-1}U^2 = C^2$.
\end{proposition}

For deciding whether the point $(2,1,1)$ on $\cS_3: B^2 + BC + 6C^2 = A^3$ 
is in the kernel of $b$ we have to look at $2T^2 + 3TU + 4U^2 = 1$. This 
equation has solutions if and only if the form $(2,3,4)$ with 
discriminant $-23$ represents $1$, hence is equivalent to 
the principal form $Q_0$. This is not the case, since 
$(2,3,4) \sim (2,-1,6)$. We may also multiply the original
equation through by $8$ and complete squares; this gives
$(4T+3U)^2 + 23U^2 = 8$. This equation is clearly unsolvable in 
integers, but has rational solutions, such as $(T,U) = (0,\frac12)$, 
for example; this implies that we do not have a chance to show the 
unsolvability of the equation using congruences or $p$-adic methods.

\subsection*{The map $b: \cS_n(\Z) \lra \Cl^+(\Delta)$ is a homomorphism.}

Consider a point $P = (A,B,C)$ on $\cS_n(\Z)$. We know that 
$\alpha = B+C\omega = \fa^n$ for some ideal $\fa$ in the maximal
order of $K$. For finding the form attached to $\fa$ we have to 
find an oriented $\Z$-basis $\{A, b+\omega\}$ of $\fa$.
Let $c$ be an integer such that $cC \equiv 1 \bmod A$; then 
$(A,B+C\omega) = (A,cB+cC\omega) = (A,\beta+\omega)$, where $\beta$
denotes an integer in the residue class $cB \equiv \frac BC \bmod A$.
It is easy to see that $\{A,\beta+\omega\}$ has the desired
properties; the form attached to $\fa$ then is
$$ Q_\fa(x,y) = \frac{N(A x + (\beta+\omega)y)}{A} = (A,2\beta+\sigma,\gamma), $$ 
where $\gamma = N(\beta + \omega)/A = Q_0(\beta,1)/A$. In particular,
$Q_\fa = Q_P$.

The map sending $\fa$ to $Q_\fa$ is known to induce an isomorphism 
between the ideal class group of $\Q(\sqrt{\Delta}\,)$ and the
strict class group of forms with discriminant $\Delta$. 


\end{document}